\documentclass[preprint,12pt]{elsarticle}

\usepackage{amssymb}
 \usepackage{graphicx}
\usepackage{epsfig}

\usepackage{amsthm}

\newcommand{\sysn}{\left\{\begin{array}{rcl}}
\newcommand{\sysk}{\end{array}\right.}

\newtheorem{theorem}{Theorem}[section]
\newtheorem{lemma}[theorem]{Lemma}

\theoremstyle{example}

\theoremstyle{definition}
\newtheorem{definition}[theorem]{Definition}

\newtheorem{problem}[theorem]{Problem}

\journal{...}

\begin{document}

\title{Covering games using semi-open sets}

\author[affil1]{Manoj Bhardwaj}

\address[affil1]{Department of Mathematics, University of Delhi, New Delhi-110007, India}

\ead[affil1]{manojmnj27@gmail.com}

\author[affil2]{Alexander V. Osipov}

\address[affil2]{Krasovskii Institute of Mathematics and Mechanics, \\ Ural Federal
 University, Yekaterinburg, Russia}

\ead[affil2]{OAB@list.ru}

\begin{abstract}

In this paper, we prove the following Theorems
\begin{enumerate}
\item An extremally disconnected space $X$ has the semi-Menger
property if and only if One does not have a winning strategy in
the game $G_{fin}(s\mathcal{O},s\mathcal{O})$. \item An extremally
disconnected space $X$ has the semi-Rothberger property if and
only if One does not have a winning strategy in the game
$G_1(s\mathcal{O},s\mathcal{O})$.
\end{enumerate}
These results answer the Problem 3.7 of \cite{H20} and Problem 3.9 of \cite{H30}.

\end{abstract}


\maketitle 


\section{Introduction}\label{sec1}

The study of topological properties via various changes is not a new idea in topological spaces. The study of selection principles in topology and their relations to game theory and Ramsey theory was started by Scheepers \cite{H1} (see also \cite{H2}). In the last two decades it has gained the enough importance to become one of the most active areas of set theoretic topology.

In 1924, Menger \cite{H4} (see also \cite{H5,H1}) introduced
Menger property in topological spaces and studied it. This
property is stronger than Lindel\"{o}fness and weaker than
$\sigma$-compactness.

Topological games form a major tool in the study of topological properties and their relations to Ramsey theory, forcing, function spaces, and other related topics. At the heart of the theory of selection principles, covering properties are defined by the ability to diagonalize, in canonical ways, sequences of open covers. Each of these covering properties has an associated two-player game. Often the nonexistence of a winning strategy for the first player in the associated game is equivalent to the original property, and this forms a strong tool for establishing results concerning the original property. We present here conceptual proofs of these type of theorems.

This paper is organized as follows. In section-2, the definitions of the terms used in this paper are provided. Section-3 deals with study of the semi-Menger game and the semi-Rothberger game.

\section{Preliminaries}\label{sec2}

Let $(X,\tau)$ or $X$ be a topological space. We will denote by $Cl(A)$ and $Int(A)$ the closure of $A$ and the interior of $A$, for a subset $A$ of $X$, respectively. Throughout this paper, $X$ stands for topological space and the cardinality of a set $A$ is denoted by $|A|$. Let $\omega$ be the first infinite cardinal and $\omega_1$ the first uncountable cardinal. The basic definitions are given.

Let $\mathcal{A}$ and $\mathcal{B}$ be collections of open covers of a topological space $X$.

The symbol $S_1(\mathcal{A}, \mathcal{B})$ denotes the selection hypothesis that for each sequence $<\mathcal{U}_n : n \in \omega>$ of elements of $\mathcal{A}$ there exists a sequence $<U_n : n \in \omega>$ such that for each $n$, $U_n \in \mathcal{U}_n$ and $\{U_n : n \in \omega\} \in \mathcal{B}$ \cite{H1}.

The symbol $S_{fin}(\mathcal{A}, \mathcal{B})$ denotes the selection hypothesis that for each sequence $<\mathcal{U}_n : n \in \omega>$ of elements of $\mathcal{A}$ there exists a sequence $<\mathcal{V}_n : n \in \omega>$ such that for each $n$, $\mathcal{V}_n$ is a finite subset of $\mathcal{U}_n$ and $\bigcup_{n \in \omega} \mathcal{V}_n$ is an element of $\mathcal{B}$ \cite{H1}.

A subset $A$ of a topological space $X$ is said to be semi-open \cite{L1} if $A\subseteq Cl(Int(A)))$.

In this paper $\mathcal{A}$ and $\mathcal{B}$ will be the collections of the following open covers of a space $X$:

$\mathcal{O}$ : the collection of all open covers of $X$,

$s\mathcal{O}$ : the collection of all semi-open covers of $X$.

\begin{definition} \label{2.1} \cite{H4}
A space $X$ is said to have \textit{Menger property} if $X$
satisfies $S_{fin}(\mathcal{O}, \mathcal{O})$.
\end{definition}

A space $X$ is said to have \textit{semi-Rothberger property}
\cite{H20} if $X$ satisfies $S_1(s\mathcal{O}, s\mathcal{O})$.

In \cite{H20}, the authors asked the following problem as Problem 3.7.

\begin{problem}
Can semi-Rothbergerness be characterized game-theoretically or Ramsey-theoretically ?
\end{problem}

\begin{definition} \cite{H30}
A space $X$ is said to have \textit{semi-Menger property} if for
each sequence $<\mathcal{U}_n : n \in \omega>$ of semi-open covers
of $X$ there is a sequence $<\mathcal{V}_n : n \in \omega>$ such
that for each $n$, $\mathcal{V}_n$ is a finite subset of
$\mathcal{U}_n$ and each $x \in X$ belongs to $\bigcup
\mathcal{V}_n$ for some $n$, i.e., $X$ satisfies
$S_{fin}(s\mathcal{O}, s\mathcal{O})$.
\end{definition}

In \cite{H30}, the authors asked the following problem as Problem 3.9.

\begin{problem}
Can semi-Mengerness be characterized game-theoretically or Ramsey-theoretically ?
\end{problem}

\section{The semi-Menger Game}

The semi-Menger game $G_{fin}(s\mathcal{O}, s\mathcal{O})$ is a game for two players, Alice and Bob, with an inning per each natural number $n$. In each inning, Alice picks a semi-open cover of the space and Bob selects finitely many members from this cover. Bob wins if the sets he selected throughout the game cover the space. If this is not the case, Alice wins.

If Alice does not have a winning strategy in the game $G_{fin}(s\mathcal{O}, s\mathcal{O})$, then $S_{fin}(s\mathcal{O}, s\mathcal{O})$ holds.
The converse implication is a deep investigation. We prove it with simplification that makes calculations easier, and an appropriate notion that goes through induction, and thus eliminates the necessity to track the history of the game.

\begin{definition}
A countable cover $\mathcal{U}$ of a space $X$ is a tail cover if the set of intersections of
cofinite subsets of $\mathcal{U}$ is an open cover of $X$.
Equivalently, a cover $\{U_1, U_2,...\}$ is a tail cover if the family
\begin{center}
$\{\bigcap_{1 \leq n \leq \infty} U_n, \bigcap_{2 \leq n \leq \infty} U_n,...\}$
\end{center}
of intersections of cofinal segments of the cover is an open cover.
\end{definition}

A tail cover is said to be s-tail cover or tail semi cover if the elements of the cover are semi-open sets.

Recall that a space $X$ is semi-Lindel\"{o}f if every semi-open
cover has a countable subcover.

It is clear that each semi-Menger space is semi-Lindel\"{o}f.

\begin{definition}(\cite{Kun})
A Hausdorff space $X$ is called a {\it Luzin space} ({\it in the
sense of Kunen})  if

(a) Every nowhere dense set in $X$ is countable;

(b) $X$ has at most countably many isolated points;

(c) $X$ is uncountable.
\end{definition}

By Corollary 2.5 in \cite{GJR}, if $X$ is an uncountable Hausdorff
space then $X$ is semi-Lindel\"{o}ff if and only if $X$ is a Luzin
space in the sense of Kunen.

\medskip

Let us note however that Kunen (Theorem 0.0. in \cite{Kun}) has
shown that under {\it Suslin's Hypothesis} (${\bf SH}$) there are
no Lusin spaces at all. K.Kunen proved that under $ {\bf
MA}(\aleph_1,\aleph_0$-centred$)$ there is a Lusin space if and
only if there is a Suslin line.

Since a Lusin space $X$ is hereditarily Lindel\"{o}f and
Hausdorff, it has cardinality at most $\mathfrak{c}=2^{\omega}$
(de Groot, \cite{Gro}).

Thus, further, we consider only hereditarily Lindel\"{o}f
Hausdorff spaces and assume that ${\bf SH}$ not hold.

\medskip

A space $X$ is an extremally disconnected space \cite{H10} if
closure of every open set is open. It is also known that in an
extremally disconnected space $X$, the collection of semi-open
sets is a topology on $X$.

\begin{theorem}
Let $X$ be an extremally disconnected space satisfying $S_{fin}(s\mathcal{O},s\mathcal{O})$ space. Then Alice does not have a winning strategy in the game $G_{fin}(s\mathcal{O},s\mathcal{O})$.
\end{theorem}
\begin{proof}
Let $\sigma$ be an arbitrary strategy for Alice in the game
$G_{fin}(s\mathcal{O},s\mathcal{O})$. If Bob covers the space
after finitely many steps in a play, then we are done. Otherwise
Bob cannot cover space after finitely many steps in each play.
Thus, we assume that, in no position, a finite selection suffices,
together with the earlier selections, to cover the space. Since
the space $X$ satisfies $S_{fin}(s\mathcal{O},s\mathcal{O})$, it
is semi-Lindel\"{o}f, we assume every semi-open cover of $X$ is
countable. By restricting Bob's moves to countable subcovers of
Alice's semi-open covers, we may assume that Alice's covers are
countable. So we can enumerate these semi-open covers as $\{U_1,
U_2,...\}$ and assume $\{U_1, \subseteq U_2,\subseteq...\}$, that
is, Alice's semi-open covers are increasing and Bob selects a
single set in each move from the increasing semi-open cover
$\{U_1, \subseteq U_2,\subseteq...\}$. Indeed, given a countable
semi-open cover $\{U_1, U_2,...\}$, we can restrict Bob's
selections to the form $\{U_1, \subseteq U_2,\subseteq...\}$, for
$n \in \omega$. Since Bob's goal is just to cover the space, we
may pretend that Bob is provided covers of the form $\{U_1, U_1
\cup U_2, U_1 \cup U_2 \cup U_3,...\}$ or making cover of space,
that is, Bob selects an element $U_1 \cup ... \cup U_n$, he
replies to Alice with the legal move $\{U_1, U_2,..., U_n\}$.
Finally, we assume that for each reply $\{U_1, U_2,...\}$(with
$U_1 \subseteq U_2 \subseteq ...$) of Alice's strategy to a move
$U$, we have $U = U_1$. Indeed, we can transform the given
semi-open cover into the semi-open cover $\{U, U \cup U_1, U \cup
U_2,...\}$. If Bob chooses $U$, we provide Alice with the answer
$U_1$, and if he chooses $U \cup U_n$, we provide Alice with the
answer $U_n$. Since Bob has already chosen the set $U$, its
addition in the new strategy does not help covering more points.
With these simplifications, Alice's strategy can be identified
with a tree of semi-open sets, as follows:
\begin{center}
Alice's initial move is a semi-open cover

$\sigma(<>) = \mathcal{U}_1 = \{U_1, U_2,...\}$;

Bob replies with $U_{k_1} = U_{\sigma(1)}$;

then Alice's move is an increasing cover

$\sigma(U_{\sigma(1)}) = \mathcal{U}_{\sigma(1)} = \{U_{\sigma(1),1}, U_{\sigma(1),2},...\}$;

Bob replies with $U_{\sigma(1),k_2} = U_{\sigma(1),\sigma(2)}$;

then Alice's move is an increasing cover

$\sigma(U_{\sigma(1),\sigma(2)}) = \mathcal{U}_{\sigma(1),\sigma(2)} = \{U_{\sigma(1),\sigma(2),1}, U_{\sigma(1),\sigma(2),2},...\}$;

Bob replies with $U_{\sigma(1),\sigma(2),k_3} = U_{\sigma(1),\sigma(2),\sigma(3)}$;

.

.

.

if Bob replies with $U_\sigma$, for $\sigma \in \mathbb{N}^m$,

then Alice's move is an increasing semi-open cover

$\sigma(U_{\sigma(1),\sigma(2),...,\sigma(m)}) = \mathcal{U}_{\sigma(1),\sigma(2),...,\sigma(m)} = \{U_{\sigma(1),\sigma(2),...,\sigma(m),1}, U_{\sigma(1),\sigma(2),...,\sigma(m),2},...\}$;

\end{center}

Now to show $\mathcal{V}_n = \bigcup_{\sigma \in \mathbb{N}^m} \mathcal{U}_\sigma$ is a tail semi-open cover of $X$.
We will show it by induction on $n$. The semi-open cover $\mathcal{V}_1 = \mathcal{U}_{<>}$ is increasing, and thus the set of cofinite intersections is again $\mathcal{V}_1$, an semi-open cover of $X$, that is, a s-tail cover of $X$.
Let $n$ be a natural number and assume it for $n$, that is, $\mathcal{V}_n$ is a tail semi-open cover of $X$. As
\begin{center}
$\mathcal{V}_n = \bigcup_{\sigma \in \mathbb{N}^m} \mathcal{U}_\sigma$

$= \bigcup_{\sigma(i) \in \omega} \mathcal{U}_{\sigma(1),\sigma(2),...,\sigma(m)}$,
\end{center}
that is, countable union of countable sets since each $\sigma(i)$ has countable infinite choices.

Since $\mathcal{V}_n$ is countable, we enumerate
\begin{center}
$\mathcal{V}_n = \{V_1, V_2,...,  V_n,...\}$

and

$\mathcal{V}_{n+1} = \{V_1 = V^1_1, V^1_2, V^1_3,...\}$ (such that $V^1_1 \subseteq V^1_2 \subseteq V^1_3 \subseteq ...$)

$\cup \{V_2 = V^2_1, V^2_2, V^2_3,...\}$ (such that $V^2_1 \subseteq V^2_2 \subseteq V^2_3 \subseteq ...$)

$\cup \{V_3 = V^3_1, V^3_2, V^3_3,...\}$ (such that $V^3_1 \subseteq V^3_2 \subseteq V^3_3 \subseteq ...$)

.

.

.

\end{center}
Now to show $\mathcal{V}_{n+1}$ is a tail semi-open cover of $X$.
Let $\mathcal{V} \subseteq \mathcal{V}_{n+1}$ is a cofinite subset. For each $k$, let $m_k$ be minimal natural number with $V^k_{m_k} \in \mathcal{V}$. Then $I = \{k \in \omega : m_k = 1\}$ is a cofinite subset of natural numbers since $\mathcal{V}$ is cofinite. Now
\begin{center}
$\cap \mathcal{V} = \cap_{k \in \omega} V^k_{m_k} = (\bigcap_{k \in I} V_k) \cap (\bigcap_{k \in \mathbb{N} \setminus I} V^k_{m_k})$.
\end{center}
The set $\bigcap_{k \in I} V_k$ is semi-open since $\mathcal{V}_n$ is a tail semi-open cover and so $\bigcap_{k \in I} V_k$ is a member of the family of intersections of cofinal segments of the semi-open cover $\mathcal{V}_n$. The set $\bigcap_{k \in \mathbb{N} \setminus I} V^k_{m_k}$ is semi-open since it is finite intersection of semi-open sets. So $\bigcap \mathcal{V}$ is a semi-open set.

Now to show $\bigcap \mathcal{V}$ covers $X$. Let $x \in X$. Since $\mathcal{V}_n$ is a tail semi-open cover, the set $I = \{k \in \omega : x \in V_k\}$ is cofinite. For each $k \in \mathbb{N} \setminus I$, let $m_k$ be minimal natural number such that $x \in V^k_{m_k}$. Then
\begin{center}
$x \in (\bigcap_{k \in I} V_k) \cap (\bigcap_{k \in \mathbb{N} \setminus I} V^k_{m_k})$,
\end{center}
the intersections of a cofinite subset of the family $\mathcal{V}_{n+1}$. So for each $n$, $\mathcal{V}_n$ is a tail semi-open cover of $X$.

Since $\mathcal{V}_n$ is a tail semi-open cover of $X$ for each $n$. Let $\mathcal{V}^{'}_n$ be the set of intersections of cofinite subsets of $\mathcal{V}_n$. Then $\mathcal{V}^{'}_n$ is a semi-open cover of $X$ for each $n$. Now by applying the property $S_{fin}(s\mathcal{O}, s\mathcal{O})$ to the sequence $\langle \mathcal{V}^{'}_n : n \in \omega \rangle$, Bob obtains $\mathcal{H}_n$, a finite subset of $\mathcal{V}^{'}_n$(an increasing cover) for each $n$ and $X = \bigcup_{n \in \omega} \bigcup \mathcal{H}_n$. Since $\mathcal{H}_n$ is finite subset of $\mathcal{V}^{'}_n$ and $\mathcal{V}^{'}_n$ is increasing cover, simply assume $\mathcal{H}_n$ is a singleton subset of $\mathcal{V}^{'}_n$. Then for $H \in \mathcal{H}_n$, $H$ is the intersection of cofinite subset say $\mathcal{W}_n$ of $\mathcal{V}_n$, that is, $H = \bigcap \mathcal{W}_n$ and
\begin{center}
$\bigcup \mathcal{H}_n = H = \bigcap \mathcal{W}_n$.
\end{center}
So $X = \bigcup_{n \in \omega} \bigcap \mathcal{W}_n$. Since $\mathcal{W}_n$ is a cofinite subset of $\mathcal{V}_n$ and in $nth$-inning, Alice provides an infinite subset $\mathcal{G}_n$(which is an increasing cover) of $\mathcal{V}_n$, Bob provides $V_n \in \mathcal{G}_n \cap \mathcal{W}_n$ and $\bigcap \mathcal{W}_n \subseteq V_n$. Then $X = \bigcup_{n \in \omega} V_n$ and Bob wins. This completes the proof.
\end{proof}

For semi-Rothberger game, we need a result slightly stronger than above Theorem. For that we first prove that semi-Menger property $S_{fin}(s\mathcal{O}, s\mathcal{O})$ is preserved by countable unions: Given a countable union of semi-Menger spaces, and a sequence of semi-open covers, we can split the sequence of covers into infinitely many disjoint subsequences, and use each subsequence to cover one of the given semi-Menger spaces.

\begin{theorem}
The semi-Menger property $S_{fin}(s\mathcal{O}, s\mathcal{O})$ is preserved by countable unions.
\end{theorem}
\begin{proof}
Let $\{X_k : k \in \omega\}$ be a family of subspaces having semi-Menger property in a space $X$ and $\langle \mathcal{U}_n : n \in \omega \rangle$ be a sequence of semi-open covers of $X$. For each $k \in \omega$, consider the sequence $\langle \mathcal{U}_n : n \geq k \rangle$. For each $k \in \omega$, since $X_k$ is semi-Menger, there is a sequence $\langle\mathcal{V}_{n,k} : n \geq k \rangle$ such that for each $n \geq k$, $\mathcal{V}_{n,k}$ is a finite subset of $\langle \mathcal{U}_n : n \geq k \rangle$ and $\bigcup_{n \geq k} \mathcal{V}_{n,k} \supseteq X_k$. For each $n$, let
\begin{center}
$\mathcal{V}_n = \bigcup \{\mathcal{V}_{n,j} : j \leq n \}$.
\end{center}
Then each $\mathcal{V}_n$ is a finite subset of $\mathcal{U}_n$. Then for each $x \in X$, $x \in \bigcup X_k$. There exists $k \in \omega$ such that $x \in X_k$. Thus, $\bigcup_{n \geq k} \mathcal{V}_{n,k} \supseteq X_k$. Then $\bigcup_{n \geq k} \mathcal{V}_{n,k} \subseteq \bigcup_{n \geq k} \mathcal{V}_n$ and hence $x \in \bigcup \mathcal{V}_n$. This completes the proof.
\end{proof}

\begin{theorem} \label{4}
Let $X$ be a space satisfying $S_{fin}(s\mathcal{O}, s\mathcal{O})$. For each strategy for
Alice in the game $G_{fin}(s\mathcal{O}, s\mathcal{O})$, there is a play according to this strategy,
\begin{center}
$(\mathcal{U}_1,\mathcal{V}_1,\mathcal{U}_2,\mathcal{V}_2,...)$,
\end{center}
such that for each point $x \in X$ we have $x \in \bigcup \mathcal{V}_n$ for infinitely many $n$.
\end{theorem}
\begin{proof}
The product space $X \times \mathbb{N}$, a countable union of semi-Menger spaces, satisfies $S_{fin}(s\mathcal{O}, s\mathcal{O})$. We define a strategy for Alice in the game $G_{fin}(s\mathcal{O}, s\mathcal{O})$, played on the space $X \times \mathbb{N}$. Let $U$ be Alice's first move in the original game. Then, in the new game, her first move is
\begin{center}
$\mathcal{U}^{'} = \{U \times \{n\} : U \in \mathcal{U}, n \in \omega\}$.
\end{center}
If Bob selects a finite set $\mathcal{V}^{'} \subseteq \mathcal{U}^{'}$, we take the set
\begin{center}
$\mathcal{V} = \{U \in \mathcal{U} : $there is $n$ with $U \times \{n\} \in \mathcal{V}^{'}\}$
\end{center}
as a move in the original game. Then Alice replies with a cover $\mathcal{V}$, and we continue in the same manner. By above Theorem, there is a play
\begin{center}
$(\mathcal{U}^{'}_1,\mathcal{V}^{'}_1,\mathcal{U}^{'}_2,\mathcal{V}^{'}_2,...)$
\end{center}
in the new game, with $\bigcup_{n \in \omega} \mathcal{V}^{'}_n$ a cover of $X \times \mathbb{N}$. Consider the corresponding play in the original strategy,
\begin{center}
$(\mathcal{U}_1,\mathcal{V}_1,\mathcal{U}_2,\mathcal{V}_2,...)$.
\end{center}
Let $x \in X$. There is a natural number $n_1$ with $(x,1) \in \bigcup \mathcal{V}^{'}_{n_1}$. Then $x \in \bigcup \mathcal{V}_{n_1}$. The set
\begin{center}
$V = \{k \in \omega :$ there is $U$ with $U \times \{k\} \in \bigcup_{i \leq {n_1}} \mathcal{V}^{'}_i\}$
\end{center}
is finite. Let $m$ be a natural number greater than all elements of the set $V$. There is a natural number $n_2$ with $(x,m) \in \bigcup \mathcal{V}^{'}_{n_2}$. Then $x \in \bigcup \mathcal{V}^{'}_{n_2}$, and $n_1 < n_2$. Continuing in a similar manner, we see that $x \in \bigcup \mathcal{V}_n$ for infinitely many $n$.
\end{proof}

\section{The Semi-Rothberger Game}

The definitions of semi-Rothberger property $S_1(s\mathcal{O}, s\mathcal{O})$ and the corresponding game $G_1(s\mathcal{O}, s\mathcal{O})$ are similar to those of $S_{fin}(s\mathcal{O}, s\mathcal{O})$ and $G_{fin}(s\mathcal{O}, s\mathcal{O})$, respectively, but here we select one element from each cover. Here too, if Alice does not have a winning strategy then the space satisfies $S_1(s\mathcal{O}, s\mathcal{O})$. The converse implication will be proved in the next results.

For that we need the following lemma.

\begin{lemma} \label{6}
Let $X$ be an extremally disconnected space satisfying $S_1(s\mathcal{O},s\mathcal{O})$. Let $\mathcal{V}_1, \mathcal{V}_2,...$ be nonempty finite families of semi-open sets such that, for each point $x \in X$, we have $x \in \bigcup \mathcal{V}_n$ for infinitely many $n$. Then there are elements $U_1 \in \mathcal{V}_1, U_2 \in \mathcal{V}_2,...$ such that the family $\{U_1, U_2,...\}$ covers the space $X$.
\end{lemma}
\begin{proof}
For each $n$, let $\mathcal{U}_n$ be the family of all intersections of $n$ semi-open sets taken from distinct members of the sequence $\mathcal{V}_1, \mathcal{V}_2,...$. Then $\mathcal{U}_n$ is a semi-open cover of $X$ for each $n$. Since $X$ satisfies $S_1(s\mathcal{O},s\mathcal{O})$, there are semi-open sets $V_1 \in \mathcal{U}_1, V_2 \in \mathcal{U}_2,...$ such that $\{V_1, V_2,...\}$ covers the space $X$. As $V_1 \in \mathcal{U}_1$, $V_1$ is a member of $\mathcal{V}_n$ for some $n$. For $V_2 \in \mathcal{U}_2$, $V_2$ is the intersection of two semi-open sets from distinct $\mathcal{V}_i$. We can extend $V_2$ to an element of some other family $\mathcal{V}_m$ and so on. From this process we obtain a selection of at most one element from each family $\mathcal{V}_i$ that covers $X$. So we can extend our selection to have an element from each family $\mathcal{V}_i$.
\end{proof}

For a natural number $k$ and families of sets $\mathcal{U}_1,...,\mathcal{U}_k$, let
\begin{center}
$\mathcal{U}_1 \wedge \mathcal{U}_2 \wedge ... \wedge \mathcal{U}_k = \{U_1 \cap U_2 \cap ... \cap U_k : U_1 \in \mathcal{U}_1,...,U_k \in \mathcal{U}_k\}$.
\end{center}

\begin{theorem}
Let $X$ be an extremally disconnected space satisfying $S_1(s\mathcal{O}, s\mathcal{O})$. Then Alice does not have a winning strategy in the game $G_1(s\mathcal{O}, s\mathcal{O})$.
\end{theorem}
\begin{proof}
Fix an arbitrary strategy for Alice in the semi-Rothberger game
$G_1(s\mathcal{O}, s\mathcal{O})$. Since $S_1(s\mathcal{O},
s\mathcal{O})$ spaces are semi-Lindel\"{o}f, we may assume that
each semi-open cover in the strategy is countable. Let
$\mathbb{N}^{< \infty}$ be the set of finite sequences of natural
numbers. We index the semi-open covers in the strategy as
\begin{center}
$\mathcal{U}_\sigma = \{U_{\sigma,1}, U_{\sigma,2},...\}$,
\end{center}
for $\sigma \in \mathbb{N}^{< \infty}$, so that $\mathcal{U} = \{U_1, U_2,...\}$ is Alice's first move, and for each finite sequence $k_1, k_2,...,k_n$ of natural numbers, $\mathcal{U}_{k_1,k_2,k_3,...,k_n}$ is Alice's reply to the position
\begin{center}
$(\mathcal{U}, U_{k_1}, \mathcal{U}_{k_1}, U_{k_1,k_2}, \mathcal{U}_{k_1,k_2},...,U_{k_1,k_2,...,k_n})$.
\end{center}
For finite sequences $\tau,\sigma \in \mathbb{N}^n$, we write $\tau \leq \sigma$ if $\tau(i) \leq \sigma(i)$ for all $i = 1,2,...,n$. We define a strategy for Alice in the semi-Menger game $G_{fin}(s\mathcal{O}, s\mathcal{O})$. Alice's first move is $\mathcal{U}$, her first move in the original strategy. Assume that Bob selects a finite subset $\mathcal{F}$ of $\mathcal{U}$. Let $m_1$ be the minimal natural number with $\mathcal{F} \subseteq \{U_1, U_2,...,U_{m_1}\}$. Then, in the semi-Menger game, Alice's response is the joint refinement $\mathcal{U}_1 \wedge \mathcal{U}_2 \wedge ... \wedge \mathcal{U}_{m_1}$. Assume that Bob chooses a finite subset $\mathcal{F}$ of this refinement. Let $m_2$ be the minimal natural number such that $\mathcal{F}$ refines all sets $\{U_{i,1},U_{i,2},...,U_{i,m_2}\}$, for $i = 1,2,...,m_1$. Then Alice's reply is the joint refinement $\bigwedge_{\tau \leq (m_1,m_2)} \mathcal{U}_{\tau}$.
In general, Alice provides a semi-open cover of the form $\bigwedge_{\tau \leq \sigma} \mathcal{U}_{\tau}$ for $\sigma \in \mathbb{N}^{< \infty}$, Bob selects a finite family refining all families $\{U_{\tau,1},U_{\tau,2},...,U_{\tau,m}\}$ for $\tau \leq \sigma$, with the minimal natural number $m$, and Alice replies $\bigwedge_{\tau \leq (\sigma,m)} \mathcal{U}_{\tau}$.

Now by Theorem \ref{4}, there is a play
\begin{center}
$(\mathcal{U}, \mathcal{F}_1, \bigwedge_{k_1 \leq m_1} \mathcal{U}_{k_1}, \mathcal{F}_2, \bigwedge_{(k_1,k_2) \leq (m_1,m_2)} \mathcal{U}_{k_1,k_2},...)$,
\end{center}
according to the new strategy, such that every point of the space is covered infinitely in the sequence $\bigcup \mathcal{F}_1, \bigcup \mathcal{F}_2,...$. By Lemma \ref{6}, we can pick one element from each set $\mathcal{F}_n$ and cover the space. There is $k_1 \leq m_1$ such that the first picked element is a subset of $U_{k_1}$. There is $k_2 \leq m_2$ such that the second picked element is a subset of $U_{k_1,k_2}$, and so on. Then the play
\begin{center}
$(\mathcal{U}, U_{k_1}, \mathcal{U}_{k_1}, U_{k_1,k_2},...)$
\end{center}
is in accordance with Alice's strategy in the semi-Rothberger game, and is won by Bob.
\end{proof}

\end{document}